\documentclass{amsart}
\usepackage{amssymb}
\usepackage{enumerate}
\usepackage{graphicx}
\usepackage[cp1250]{inputenc}
\usepackage[OT4]{fontenc}
\def\CC{\mathbb{C}}
\def\RR{\mathbb{R}}
\def\ZZ{\mathbb{Z}}
\def\NN{\mathbb{N}}
\def\DD{\mathbb{D}}
\def\TT{\mathbb{T}}
\def\CDD{\overline{\DD}}
\def\wi{\widetilde}

\def\pa{\partial}

\def\su{\subset}
\def\cd{\CC\setminus[1,\infty)}
\def\co{\CC\setminus[0,\infty)}
\def\OO{{\mathcal O}}
\def\cC{{\mathcal C}}
\def\eps{\varepsilon}
\def\lon{\longrightarrow}
\def\sm{\setminus}
\DeclareMathOperator{\re}{Re}
\DeclareMathOperator{\im}{Im}
\DeclareMathOperator{\dist}{dist}
\DeclareMathOperator{\inte}{int}
\DeclareMathOperator{\Arg}{Arg}
\DeclareMathOperator{\Log}{Log}
\DeclareMathOperator{\res}{Res}
\DeclareMathOperator{\ord}{ord}
\renewcommand{\phi}{\varphi}
\newtheorem{tw}{Theorem}[section]
\newtheorem{lem}[tw]{Lemma}
\newtheoremstyle{rem}{}{}{\it}{}{\bf}{.}{ }{}
\theoremstyle{rem}
\newtheorem{rem}[tw]{Remark}
\newtheoremstyle{pro}{}{}{\it}{}{\bf}{.}{ }{}
\theoremstyle{pro}
\newtheorem{pro}[tw]{Proposition}
\begin{document}
\title{Extension of holomorphic functions onto a special domain}
\author{Tomasz Warszawski}
\address{Institute of Mathematics, Faculty of Mathematics and Computer Science, Jagiellonian University, ul. Prof. St. \L ojasiewicza 6, 30-348 Krak\'ow, Poland}
\email{tomasz.warszawski@im.uj.edu.pl}
\thanks{This is a semester paper prepared under the guidance of Professor Pawe\l\ Doma\'nski in the framework of Ph. D. program \'Srodowiskowe Studia Doktoranckie z Nauk Matematycznych}
\begin{abstract}
We present a modified version of the Arakelyan's result: a relationship between holomorphic extension of a holomorphic function on the unit disc onto the domain $\cd$ and its Taylor coefficients' interpolation.
\end{abstract}
\maketitle
This paper aims at explaining some ideas and simplifying the proof from the paper of N. U. Arakelyan (see \cite{Ara}) in the case of the domain $\cd$. The problem is to characterize these holomorphic functions on the unit disc, which extend holomorphically on $\cd$. There are known several such conditions, all connected with the existence of a holomorphic function on the half-plane or the plane, interpolating Taylor coefficients of the given function and having controlled growth at infinity. This growth is measured by the (inner) exponential type, the indicator function or the order.
\section{Notation and results}
By $\DD$ we denote the unit open disc on the complex plane and by $\TT$ --- the unit circle. Moreover, let \begin{eqnarray*}H&:=&\{z\in\CC:\re z\geq 0\},\\\Delta(\theta_1,\theta_2)&:=&\{z\in\CC:\theta_1\leq\arg z\leq\theta_2\}\cup\{0\},\quad-\pi\leq\theta_1\leq\theta_2\leq\pi,\end{eqnarray*} where $$\arg:\CC\sm\{0\}\longrightarrow(-\pi,\pi],\text{ continuous in }\CC\sm(-\infty,0],$$ is the main argument. Sets $\Delta(\theta_1,\theta_2)$ we shall call \emph{closed sectors}. Let $\log$ be the standard branch of the logarithm (continuous in $\CC\sm(-\infty,0]$) that is $\log z:=\log|z|+i\arg z$, $z\neq 0$. Put $\log 0:=-\infty$ and $\log^+t:=\max\{\log t,0\}$ for $t\geq 0$. The interior of a set $S\su\CC$ we denote by $\inte S$. 

Let $G\su\CC$ be a closed unbounded connected set and let a function $\phi$ be holomorphic in an open neighborhood of $G$ (which we denote $\phi\in\OO(G)$). The \emph{exponential type} of $\phi$ on $G$ is defined as \begin{equation}\label{et} ET_G(\phi):=\limsup_{z\in G,\ z\to\infty}\frac{\log^+|\phi(z)|}{|z|}=\max\left\{\limsup_{z\in G,\ z\to\infty}\frac{\log|\phi(z)|}{|z|},\ 0\right\}.\end{equation} Note that \begin{equation}\label{5} ET_G(\phi+\psi)\leq\max\{ET_G(\psi),ET_G(\phi)\},\quad\phi,\psi\in\OO(G).\end{equation}

Assume that $D=\inte\Delta\neq\emptyset$ for some closed sector $\Delta$ and $\phi\in\OO(D)$. Then the \emph{inner exponential type} of $\phi$ on $D$ we define by \begin{equation}\label{iet} IET_D(\phi):=\sup\{ET_{\wi\Delta}(\phi):\wi\Delta\su D\cup\{0\}\text{ and }\wi\Delta\text{ is a closed sector}\}.\end{equation} We easily see that $$IET_D(\phi+\psi)\leq\max\{IET_D(\phi),IET_D(\psi)\},\quad\phi,\psi\in\OO(D),$$ and $$IET_D(\phi)\leq ET_{\Delta}(\phi),\quad\phi\in\OO(\Delta).\bigskip\bigskip$$

We are ready to formulate 
\begin{tw}[Arakelyan, \cite{Ara}, Theorem 1.1]\label{tw1}
Let a function $f(z)=\sum_{n=0}^\infty a_nz^n$ be holomorphic in $\DD$. Then $f$ extends holomorphically on $\cd$ if and only if there exists a function $\phi\in\OO(H)$ such that 
\begin{enumerate}[$(i)$]
\item $\phi(n)=a_n,\ n\geq 0;$
\item $IET_{\inte H}(\phi)=0$.\\
\end{enumerate}
\end{tw}
For a sector $\Delta=\Delta(\theta_1,\theta_2)$ and a function $\phi\in\OO(\Delta)$, we introduce the \emph{indicator function} $$h_\phi(\theta):=\limsup_{R>0,\ R\to\infty}\frac{\log|\phi(Re^{i\theta})|}{R},\quad\theta\in[\theta_1,\theta_2].$$ Similarly as before, \begin{equation}\label{6} h_{\phi+\psi}\leq\max\{h_\phi,h_\psi\},\quad\phi,\psi\in\OO(\Delta).\bigskip\bigskip\end{equation}

Using this concept allows us to state
\begin{tw}[Arakelyan, \cite{Ara}, Theorem 1.1]\label{tw2}
Let a function $f(z)=\sum_{n=0}^\infty a_nz^n$ be holomorphic in $\DD$. Then $f$ extends holomorphically on $\cd$ if and only if there exists a function $\phi\in\OO(H)$ such that 
\begin{enumerate}[$(i)$]
\item $\phi(n)=a_n,\ n\geq 0;$
\item $IET_{\inte H}(\phi)<\infty;$
\item $h_\phi(\theta)\leq 0,\ |\theta|<\pi/2$.\\
\end{enumerate}
\end{tw}
The \emph{order} of an entire function $\phi$ is defined as $$\ord\phi:=\inf\{c\geq 0:|\phi(z)|\leq e^{|z|^c}\text{ as }z\to\infty\}.$$ We have $$\ord\phi=\limsup_{R>0,\ R\to\infty}\frac{\log\log\max\{|\phi(z)|:|z|=R\}}{\log R},\quad\phi\neq\text{const}.$$ 

If $\Delta$ is a closed sector with non-empty interior, we call a domain $D\su\Delta$ \emph{asymptotic} to $\Delta$ if $$\lim_{z\in\Delta\setminus D,\ z\to\infty}\frac{\dist(z,\pa\Delta)}{|z|}=0.$$

Note that the definition of $IET_D(\phi)$ extends in the case of such a domain $D$ and functions $\phi\in\OO(D)$, where the `sup' in \eqref{iet} is taken over all closed sectors $\wi\Delta\su\inte\Delta\cup\{0\}$. Indeed, for any such $\wi\Delta$ there is $R>0$ such that $\wi\Delta\sm R\DD\su D$, so the sup-limit in \eqref{et} still makes sense.\bigskip\bigskip

The following theorem is the most general. We will prove that Theorems \ref{tw1} and \ref{tw2} follow from it.
\begin{tw}[Arakelyan, \cite{Ara}, Theorem 1.2]\label{tw3}
In order that a holomorphic function $f(z)=\sum_{n=0}^\infty a_nz^n$, $z\in\DD$, extend holomorphically on $\cd$, the following condition is sufficient.

There exist a domain $D$ asymptotic to $H$, a function $\phi\in\OO(D)$ and a number $n_0\in\NN$ such that 
\begin{enumerate}[$(i)$]
\item $\phi(n)=a_n,\ n\geq n_0;$
\item $IET_{D}(\phi)<\infty;$
\item $h_\phi(\theta)\leq 0,\ |\theta|<\pi/2$.\\
\end{enumerate}
The following condition is necessary.

There exists an entire function $\phi$ such that
\begin{enumerate}[$(i)$]
\setcounter{enumi}{3}
\item $\phi(n)=a_n,\ n\geq 0;$
\item $\phi(n)=0,\ n=-1,-2,\ldots;$
\item $ET_{H}(\phi)<\infty;$
\item \label{3} $h_\phi(\theta)\leq 0,\ |\theta|<\pi/2;$
\item $\ord\phi\leq 1$.\\
\end{enumerate}
\end{tw}
\smallskip
\begin{rem} 
The conditions $h_\phi(0)\leq 0$ and $\phi(n)=a_n$ for big $n\in\NN$, imply that a formal series  $\sum_{n=0}^\infty a_nz^n$ is convergent in $\DD$. 
\end{rem}
Indeed, we have $$\limsup_{n\to\infty}\frac{\log|a_n|}{n}\leq 0$$ or equivalently $$\limsup_{n\to\infty}|a_n|^{1/n}\leq 1.$$
\medskip
\begin{rem}
Theorem \ref{tw2} follows from Theorem \ref{tw3}. 
\end{rem}
Actually, it suffices to take $D:=\inte H$.
\bigskip
\begin{pro}\label{pro1}
Theorem \ref{tw1} follows from Theorem \ref{tw2}.
\end{pro}
To see it, we shall use the following
\begin{tw}[Levin, \cite{Lev}, Chapter 1, \S 18, Theorem 28, case $\rho(r)\equiv\rho=1$]\label{tw5}
Let $\Delta=\Delta(\theta_1,\theta_2)$ be a closed sector and let a function $\phi\in\OO(\Delta)$ be such that \begin{equation}\label{13}\rho:=\inf\{c\geq 0:|\phi(z)|\leq e^{|z|^c}\text{ for }z\in\Delta,\ z\to\infty\}\end{equation} is equal to $1$. Then for any $\eps>0$ there is $r_\eps>0$ such that $$\frac{\log|\phi(re^{i\theta})|}{r}\leq h_\phi(\theta)+\eps,\quad r\geq r_\eps,\quad\theta_1\leq\theta\leq\theta_2.$$
\end{tw}
\bigskip
\begin{proof}[Proof of Proposition \ref{pro1}]
Let the conditions $(i),(ii)$ of Theorem \ref{tw1} be satisfied. Taking $\Delta=\Delta(\theta,\theta)$, $|\theta|<\pi/2$, we get $$h_\phi(\theta)\leq ET_\Delta(\phi)\leq IET_{\inte H}(\phi)=0,$$ so by Theorem \ref{tw2}, the function $f$ extends holomorphically on $\cd$.

Now, let $f$ extend holomorphically on $\cd$. Let $\Delta\su\inte H\cup\{0\}$ be a closed sector and let $\eps>0$. The condition $(ii)$ of Theorem \ref{tw2} implies $ET_\Delta(\phi)<\infty$, so any number $c>1$ fulfills $|\phi(z)|\leq e^{|z|^c}$ for sufficiently big $z\in\Delta$. Hence the number $\rho$ defined by \eqref{13} is not greater than $1$. 

If $\rho<1$ then for any $c\in(\rho,1)$ we have $|\phi(z)|\leq e^{|z|^c}$ if $z\in\Delta$ is sufficiently big, so $$\limsup_{z\in\Delta,\ z\to\infty}\frac{\log|\phi(z)|}{|z|}\leq\limsup_{z\in\Delta,\ z\to\infty}\frac{|z|^c}{|z|}=0.$$ Hence $ET_\Delta(\phi)=0$ and $IET_{\inte H}(\phi)=0$, as claimed.

If $\rho=1$ then in virtue of the condition $(iii)$ of Theorem \ref{tw2}, it follows from Theorem \ref{tw5} that $$\frac{\log|\phi(z)|}{|z|}\leq\eps,\quad z\in\Delta,\ |z|\geq r_\eps.$$ Thus $$\limsup_{z\in\Delta,\ z\to\infty}\frac{\log|\phi(z)|}{|z|}\leq 0,$$ so $ET_\Delta(\phi)=0$ and $IET_{\inte H}(\phi)=0$.
\end{proof}
Therefore, all we need is to prove Theorem \ref{tw3}.
\bigskip\bigskip
\section{Proof of Theorem 1.3 (sufficiency)}
Divide the given series into two parts $$\sum_{n=0}^\infty a_nz^n=\sum_{n=0}^m a_nz^n+\sum_{n=m+1}^\infty\phi(n)z^n,\quad m\geq n_0.$$ 

Let $\Arg:\CC\sm\{0\}\lon[0,2\pi)$ be a branch of the argument, continuous in $\co$, and let $$z^\zeta:=e^{\zeta(\log|z|+i\Arg z)},\quad z\neq 0,\ \zeta\in\CC.$$

We claim that for any compact set $K\su\co$, $\inte K\neq\emptyset$, there is $m\in\NN$ and a contour $\Gamma_m\su D$ such that \begin{equation}\label{2}\sum_{n=m+1}^\infty\phi(n)z^n=\int_{\Gamma_m}\underbrace{\frac{\phi(\zeta)z^\zeta}{e^{2\pi i\zeta}-1}}_{=:g(\zeta,z)}d\zeta=:I_m(z),\quad z\in\DD\cap K\end{equation} and $I_m(z)$ converge uniformly in $K$ to a function, which restricted to $\inte K$ is holomorphic.

We define $$a:=\min_{z\in K}(\pi-|\Arg z-\pi|)>0,\quad b:=\max_{z\in K}\log|z|.$$ Let $\theta\in(0,\pi/2)$ satisfy $b\cot\theta-a<0$ (the number $b\cot\theta-a$ will appear in some inequality). Put $\Delta:=\Delta(-\theta,\theta)\su\inte H\cup\{0\}$. 

In what follows, numbers $m$ are integers. Let $$G_m:=\Delta\sm\DD(0,m+1/2),\quad\Gamma_m:=\pa G_m,\quad\gamma_m:=\Gamma_m\cap\pa\DD(0,m+1/2),\quad m>0,$$ where $\DD(a,r)$ is an open disc with a radius $r>0$, centered at a point $a\in\CC$, and $\Gamma_m$ is positively oriented w.r.t. $G_m$.

For sufficiently big $m\in\NN$, namely for $m\geq m_0$, the domain $D$ contains $G_{m}$ (as $D$ is asymptotic to $H$). Let \begin{equation}\label{21}\eps(\zeta):=\frac{\log|\phi(\zeta)|}{|\zeta|},\quad\zeta\in G_{m_0}.\end{equation} The conditions $(ii),(iii)$ of Theorem \ref{tw3} combined with Theorem \ref{tw5} give (cf. the proof of Proposition \ref{pro1}) $$\limsup_{\zeta\to\infty}\eps(\zeta)\leq 0.$$ 

Note that \begin{multline}\label{20}|z^\zeta|=\exp(\re(\zeta\log|z|+i\zeta\Arg z))=\exp(\log|z|\re\zeta-\Arg z\im\zeta))\\\leq\exp(\log|z|\re\zeta+|\Arg z-\pi||\im\zeta|-\pi\im\zeta),\quad\zeta\in G_{m_0},\ z\in\co.\end{multline}

There is $r=r(\Delta)\in(0,1/2)$ such that $\DD(m,r)\su G_{m_0}$ for $m>m_0$. Furthermore, \begin{equation}\label{1}\frac{1}{|e^{2\pi i\zeta}-1|}\leq c\exp(\pi\im\zeta-\pi|\im\zeta|),\quad\zeta\in\CC\sm\bigcup_{m\in\ZZ}\DD(m,r).\end{equation}

Indeed, if $\im\zeta<-1$ then $$|e^{2\pi i\zeta}-1|\geq|e^{2\pi i\zeta}|-1=e^{-2\pi\im\zeta}-1>0$$ and $$\pi\im\zeta-\pi|\im\zeta|=2\pi\im\zeta,$$ so $$\frac{1}{|e^{2\pi i\zeta}-1|}\leq\frac{1}{e^{-2\pi\im\zeta}-1}=\frac{e^{2\pi\im\zeta}}{1-e^{2\pi\im\zeta}}\leq
\frac{1}{1-e^{-2\pi}}e^{2\pi\im\zeta}.$$

For $\im\zeta>1$ we estimate $$|e^{2\pi i\zeta}-1|\geq 1-|e^{2\pi i\zeta}|=1-e^{-2\pi\im\zeta}>0$$ and $$\pi\im\zeta-\pi|\im\zeta|=0,$$ so $$\frac{1}{|e^{2\pi i\zeta}-1|}\leq\frac{1}{1-e^{-2\pi\im\zeta}}\leq
\frac{1}{1-e^{-2\pi}}.$$ Concluding, the inequality \eqref{1} holds with a constant $(1-e^{-2\pi})^{-1}$ on the set $\{\zeta\in~\CC:|\im\zeta|>1\}$.

It remains to consider the case $|\im\zeta|\leq 1$. Note that both sides of \eqref{1} do not change after adding to $\zeta$ any integer. Hence one may assume without loss of generality that $|\re\zeta|\leq 1$. The set $$Q:=\{\zeta\in\CC:|\re\zeta|\leq 1,\ |\im\zeta|\leq 1\}\sm\bigcup_{m\in\ZZ}\DD(m,r)$$ is compact and the function $$q:Q\ni\zeta\longmapsto|e^{2\pi i\zeta}-1|\exp(\pi\im\zeta-\pi|\im\zeta|)\in[0,\infty)$$ has no zeroes. Therefore $c':=\inf_Qq>0$, so the inequality \eqref{1} is satisfied on $Q$ with a constant $1/c'$.

Finally, \eqref{1} holds on the set $\CC\sm\bigcup_{m\in\ZZ}\DD(m,r)$ with $$c:=\max\left\{\frac{1}{1-e^{-2\pi}},\ \frac{1}{c'}\right\}.$$

Multiplying \eqref{21}, \eqref{20} and \eqref{1} by sides, we obtain for $\zeta\in G_{m_0}\sm\bigcup_{m>m_0}\DD(m,r)$ and $z\in\co$ \begin{equation}|\label{11}g(\zeta,z)|\leq c\exp(\log|z|\re\zeta-(\pi-|\Arg z-\pi|)|\im\zeta|+\eps(\zeta)|\zeta|).\end{equation} This implies for $\zeta\in\Gamma_{m}\sm\gamma_{m}$, $m>m_0$ and $z\in K$ \begin{eqnarray*}|g(\zeta,z)|&\leq&c\exp(b|\zeta|\cos\theta-a|\zeta|\sin\theta+\eps(\zeta)|\zeta|)\\&=&
c\exp([(b\cot\theta-a)\sin\theta+\eps(\zeta)]|\zeta|).
\end{eqnarray*} As $(b\cot\theta-a)\sin\theta$ is a negative constant, say $c'$, and $\limsup_{\zeta\to\infty}\eps(\zeta)\leq 0$, we may estimate for big $m$ $$\left|\int_{\Gamma_m\sm\gamma_m}g(\zeta,z)d\zeta\right|\leq\int_{\Gamma_m\sm\gamma_m}c\exp\left(\frac{c'}{2}|\zeta|\right)|d\zeta|
\leq\int_{0}^\infty c\exp\left(\frac{c'}{2}\rho\right)d\rho<\infty.$$ Hence the integrals $I_m(z)$ converge uniformly in $K$, consequently they are holomorphic in $\inte K$.

For $z\in\DD\cap K$, $\zeta\in\gamma_m$ and $m>m_0$, it follows from \eqref{11} that \begin{eqnarray*}|g(\zeta,z)|&\leq& c\exp(\log|z||\zeta|\cos\theta+\eps(\zeta)|\zeta|)\\&=&
c\exp([\log|z|\cos\theta+\eps(\zeta)]|\zeta|).\end{eqnarray*} Since $c'':=\log|z|\cos\theta$ is a negative constant, we have for $z\in\DD\cap K$ \begin{multline*}\left|\int_{\gamma_m}g(\zeta,z)d\zeta\right|\leq\int_{\gamma_m}c\exp\left(\frac{c''}{2}|\zeta|\right)|d\zeta|\\
\\\leq 2\pi c\exp\left(\frac{c''}{2}\left(m+\frac{1}{2}\right)\right)\left(m+\frac{1}{2}\right)\to 0,\quad m\to\infty.\end{multline*}

By the Residue Theorem, applied to the contours $\pa(G_m\sm G_n)$, $n>m\geq m_0+1$, we get passing with $n$ to infinity $$\int_{\Gamma_m}g(\zeta,z)d\zeta=2\pi i\sum_{n=m+1}^\infty\res(g(\cdotp,z),n),\quad z\in\DD\cap K,$$ where $\res(F,x)$ is a residue of a function $F$ at a point $x$. Clearly, $$\res(g(\cdotp,z),n)=\lim_{\zeta\to n}(\zeta-n)g(\zeta,z)=\frac{\phi(n)z^n}{2\pi i},\quad z\in\DD\cap K,$$ whence \eqref{2} is proved.

To finish the proof it suffices to use the Monodromy Theorem for the simply connected domain $\co$. More precisely, we fix any compact curve $\gamma\su\co$ such that $\gamma\cap\DD\neq\emptyset$. Let $K\su\co$ be a compact set such that $\gamma\su\inte K$. Therefore, $$\sum_{n=0}^\infty a_nz^n=\sum_{n=0}^m a_nz^n+I_m(z)\ \ \text{for\, }z\in\DD\cap\gamma\ \text{and big } m\in\NN,$$ but the right side extends holomorphically on $\inte K$, being an open neighborhood of $\gamma$.\qed
\bigskip\bigskip
\section{Proof of Theorem 1.3 (necessity)}
In the whole proof the following notation is valid. Let $$\gamma_{r,\theta}(t):=te^{i\theta},\quad t\in[r,\infty),$$ for parameters $r\geq 0$ and $\theta\in[-\pi,\pi)$. Denote by $\gamma_{r,\theta}^-$ the reversely oriented $\gamma_{r,\theta}$. We identify these curves with their images, that is half-lines. We assume that any circle $r\TT$, where $r>0$, is positively oriented w.r.t. $r\DD$. Let $$D_{r,\theta}:=\CC\sm(r\CDD\cup\gamma_{r,\theta})$$ and $$\Log_{\theta}\zeta:=\log|z|+i\Arg_{\theta}\zeta,$$ where $$\Arg_{\theta}:\CC\sm\{0\}\lon[\theta,\theta+2\pi)$$ is a branch of the argument, continuous in $D_{0,\theta}$. In particular, $\Arg_\theta e^{i\theta}=\theta$. We fix an arbitrary $\eta\in(0,1/2)$.
\begin{lem}\label{9}
Let a function $g(z)=\sum_{n=0}^\infty a_nz^n$ be holomorphic in $\DD$ and let it extend holomorphically on $\cd$. Assume that $\theta\in[-\pi,\pi)\sm\{0\}$ and $$|g(\zeta)|<Me^{-|\zeta|^\eta},\quad\zeta\in\gamma_{0,\theta}$$ with some constant $M$. Then for $r\in(0,1)$ the formula \begin{equation}\label{14}\phi_{r,\theta}(z):=\frac{1}{2\pi i}\int_{\gamma_{r,\theta}^-\cup r\TT}e^{(-z-1)\Log_{\theta}\zeta}g(\zeta)d\zeta+\frac{1}{2\pi i}\int_{\gamma_{r,\theta}}e^{(-z-1)(\Log_{\theta}\zeta+2\pi i)}g(\zeta)d\zeta\end{equation} defines an entire function such that 
\begin{itemize}
\item $\phi_{r,\theta}$ does not depend on $r\in(0,1)$ and we denote $\phi_{\theta}:=\phi_{r,\theta};$
\item $\phi_{\theta}(n)=a_n,\ n\geq 0;$
\item $\phi_{\theta}(n)=0,\ n=-1,-2,\ldots;$
\item $ET_{H}(\phi_{\theta})<\infty;$
\item $\ord\phi_{\theta}\leq 1$.\\
\end{itemize}

If additionally $g$ is entire then the formula \eqref{14} defines for $r>0$ {$($}$\theta$ can be equal to 0$)$ an entire function such that 
\begin{itemize}
\item $\phi_{r,\theta}$ does not depend on $r>0;$ 
\item $h_{\phi_{\theta}}(\wi\theta)=-\infty,\ |\wi\theta|<\pi/2$.\\
\end{itemize}
\end{lem}

\begin{proof}
Let $$e^{(-z-1)\Log_{\theta}\zeta}g(\zeta)=:\alpha_{z,\theta}(\zeta),\quad e^{(-z-1)(\Log_{\theta}\zeta+2\pi i)}g(\zeta)=:\beta_{z,\theta}(\zeta).$$

First, we will show that the integrals are locally uniformly convergent. This fact is obvious for the integral over the circle $r\TT$. Since $$\int_{\gamma_{r,\theta}}\beta_{z,\theta}=e^{-2\pi iz}\int_{\gamma_{r,\theta}}\alpha_{z,\theta},$$ it suffices to show that the integrals $\int_{\gamma_{r,\theta}}\alpha_{z,\theta}$ are locally uniformly convergent. To do this, we estimate for $z\in\CC$ and $r\in(0,1)$ \begin{eqnarray*}\int_{\gamma_{r,\theta}}|\alpha_{z,\theta}(\zeta)||d\zeta|&=&\int_r^\infty|
e^{(-z-1)(\log\varrho+i\theta)}g(\varrho e^{i\theta})e^{i\theta}|d\varrho\\&\leq&\int_r^\infty\varrho^{-\re z-1}e^{\theta\im z}Me^{-\varrho^\eta}d\varrho\\&\leq& Me^{\theta\im z}\left(\int_r^1\varrho^{-\re z-1}e^{-\varrho^\eta}d\varrho+\int_1^\infty\varrho^{-\re z-1}e^{-\varrho^\eta}d\varrho\right)\\&\leq&
Me^{\theta\im z}\left(\int_r^1\varrho^{-|z|}\varrho^{-1}d\varrho+
\int_1^\infty\varrho^{|z|-1}e^{-\varrho^\eta}d\varrho\right)\\&\leq& Me^{\theta\im z}\left(\int_r^1r^{-|z|}\varrho^{-1}d\varrho+
\frac{1}{\eta}\int_1^\infty\rho^{\frac{|z|}{\eta}-1}e^{-\rho}d\rho\right)\\&=&Me^{\theta\im z}\left(r^{-|z|}\log(1/r)+\frac{1}{\eta}\int_1^\infty\rho^{\frac{|z|}{\eta}-1}e^{-\rho}d\rho\right).\end{eqnarray*} Now, if $|z|<\eta$ then the last integral is majorized by $$\frac{1}{\eta}\int_1^\infty e^{-\rho}d\rho=\frac{1}{\eta e}.$$ Otherwise, if $|z|\geq\eta$ we use the gamma function $$\Gamma(\zeta):=\int_0^\infty\rho^{\zeta-1}e^{-\rho}d\rho,\quad\re\zeta>0,$$ to get $$\frac{1}{\eta}\int_1^\infty\rho^{\frac{|z|}{\eta}-1}e^{-\rho}d\rho\leq\frac{1}{\eta}\Gamma\left(\frac{|z|}{\eta}\right).$$ From the Stirling's formula there exist constants $c',C'>0$ such that $$\Gamma\left(\frac{|z|}{\eta}\right)\leq c'\exp(c'|z|\log|z|),\quad|z|>C'.$$ In the remaining case $\eta\leq|z|\leq C'$ we estimate $\frac{1}{\eta}\Gamma\left(\frac{|z|}{\eta}\right)$ just by a constant.

The above three cases can be put together by one inequality (with a constant $c$) $$\frac{1}{\eta}\int_1^\infty\rho^{\frac{|z|}{\eta}-1}e^{-\rho}d\rho\leq c\exp(c|z|\log(|z|+1)),\quad z\in\CC.$$ For $z\in\CC$ and $r\in(0,1)$ it follows that \begin{equation}\label{18}\left|\int_{\gamma_{r,\theta}}\alpha_{z,\theta}\right|\leq Me^{\theta\im z}\left(r^{-|z|}\log(1/r)+c\exp(c|z|\log(|z|+1))\right).
\end{equation}

Concluding, $\phi_{r,\theta}$ are entire functions. To show that for any $z\in\CC$, the number $\phi_{r,\theta}(z)$ does not depend on $r$, take $0<r<R<1$ and consider the difference $$2\pi i(\phi_{R,\theta}(z)-\phi_{r,\theta}(z))=\int_{R\TT}\alpha_{z,\theta}+\int_{R e^{i\theta}}^{r e^{i\theta}}\beta_{z,\theta}-\int_{r\TT}\alpha_{z,\theta}+\int_{r e^{i\theta}}^{R e^{i\theta}}\alpha_{z,\theta}.$$ The above sum is a limit of a sequence of integrals of the only one function $\alpha_{z,\theta}$, over closed paths contained in $D_{0,\theta}$ (more precisely, such path is the boundary of an annulus of radiuses $r$ and $R$, from which a figure approximating the interval between points $re^{i\theta}$ and $Re^{i\theta}$ is removed). By the Cauchy Theorem the integral of $\alpha_{z,\theta}$ over any such path vanishes, so the limit vanishes, as well.

As in the assertion we denote the common value of $\phi_{r,\theta}(z)$ by $\phi_\theta(z)$. We have \begin{multline}\label{12}2\pi i\phi_\theta(z)=\int_{r\TT}\alpha_{z,\theta}-\int_{\gamma_{r,\theta}}\alpha_{z,\theta}+\int_{\gamma_{r,\theta}}e^{-2\pi iz}\alpha_{z,\theta}\\=\int_{r\TT}e^{(-z-1)\Log_{\theta}\zeta}g(\zeta)d\zeta+(e^{-2\pi iz}-1)\int_{\gamma_{r,\theta}}e^{(-z-1)\Log_{\theta}\zeta}g(\zeta)d\zeta,\quad z\in\CC,\ r\in(0,1),\end{multline} so for $z$ equal to the integer $n$, the number $\phi_\theta(n)$ is the $n$th coefficient of the Laurent series of $g$, that is $\phi_\theta(n)=a_n$ for $n\geq 0$ and $\phi_\theta(n)=0$ for $n<0$.

To obtain estimates of the order, the exponential type and the indicator function we also need some inequalities for the integral over $r\TT$, namely \begin{multline}\label{15}\left|\int_{r\TT}\alpha_{z,\theta}\right|\leq\int_{\theta}^{\theta+2\pi}|e^{(-z-1)(\log r+it)}g(re^{it})rie^{it}|dt\\
\leq\max_{r\TT}|g|\,r^{-\re z}\int_{\theta}^{\theta+2\pi}e^{t\im z}dt\leq 2\pi\max_{r\TT}|g|\,r^{-\re z}e^{3\pi|\im z|},\quad z\in\CC,\ r\in(0,1).\end{multline} 

Combining \eqref{12} and \eqref{15}, we arrive at 
\begin{multline}\label{23}2\pi|\phi_\theta(z)|\leq\left|\int_{r\TT}\alpha_{z,\theta}\right|+|e^{-2\pi iz}-1|\left|\int_{\gamma_{r,\theta}}\alpha_{z,\theta}\right|\\\leq C_rr^{-\re z}e^{3\pi|\im z|}+2e^{2\pi|\im z|}\left|\int_{\gamma_{r,\theta}}\alpha_{z,\theta}\right|,\quad z\in\CC,\ r\in(0,1),\end{multline} for a constant $C_r$ depending only on $r$. 

From \eqref{23} and \eqref{18} we get \begin{multline}\label{16}|\phi_\theta(z)|\leq c_{(r)}\exp(\re z\log(1/r)+3\pi|\im z|)+\\+c_{(r)}e^{3\pi|\im z|}(e^{|z|\log(1/r)}\log(1/r)+e^{c|z|\log|z|})\leq C_{(r)}\exp(C_{(r)}|z|\log|z|)\end{multline} for big $z\in\CC$ and constants $c_{(r)},C_{(r)}$ depending only on $r$, so $\ord\phi_\theta\leq 1$. 

Note that for $z\in H$, i.e. $\re z\geq 0$, we can estimate, using some of the inequalities obtained at the beginning for $\int_{\gamma_{r,\theta}}|\alpha_{z,\theta}(\zeta)||d\zeta|$ \begin{multline*}\left|\int_{\gamma_{r,\theta}}\alpha_{z,\theta}\right|\leq Me^{\theta\im z}\left(\int_r^1\varrho^{-\re z-1}e^{-\varrho^\eta}d\varrho+\int_1^\infty\varrho^{-\re z-1}e^{-\varrho^\eta}d\varrho\right)\\\leq Me^{\theta\im z}\left(r^{-|z|}\log(1/r)+\int_1^\infty e^{-\varrho^\eta}d\varrho\right),\quad\ r\in(0,1),\end{multline*} but $C:=\int_1^\infty e^{-\varrho^\eta}d\varrho$ is a constant. Therefore, we obtain (changing $e^{c|z|\log|z|}$ to $C$ in \eqref{16}) $$|\phi_\theta(z)|\leq c_{[r]}e^{C_{[r]}|z|},\quad z\in H,\,z\to\infty,$$ with constants $c_{[r]},C_{[r]}$ depending only on $r$, which shows that $ET_H(\phi_\theta)\leq C_{[r]}$ (in fact, $C_{[r]}\to 3\pi$ as $r\to 1$, so $ET_H(\phi_\theta)\leq 3\pi$).

\bigskip\bigskip
Assume that $g$ is entire. The local uniform convergence of the integrals follows from $$\int_{\gamma_{r,\theta}}|\alpha_{z,\theta}(\zeta)||d\zeta|\leq\int_{\gamma_{1/2,\theta}}|\alpha_{z,\theta}(\zeta)||d\zeta|.$$ The proof that $\phi_{r,\theta}$ does not depend on $r>0$ goes in the same way as before. We have to verify only the equality for the indicator function.

Fix $\wi\theta\in(-\pi/2,\pi/2)$ and let $z=Re^{i\wi\theta}$, $R>0$. Since $\re z>0$, we may proceed as follows for $r>0$ \begin{eqnarray*}\left|\int_{\gamma_{r,\theta}}\alpha_{z,\theta}\right|&\leq&
\int_r^\infty|e^{(-z-1)(\log\varrho+i\theta)}g(\varrho e^{i\theta})e^{i\theta}|d\varrho\\&\leq&Me^{\theta\im z}\int_r^\infty\varrho^{-\re z-1}e^{-\varrho^\eta}d\varrho\\&\leq&Me^{\theta\im z}\int_r^\infty r^{-\re z-1}e^{-\varrho^\eta}d\varrho\\&\leq&M'e^{\theta\im z}r^{-\re z-1}\end{eqnarray*} with a constant $M':=M\int_0^\infty e^{-\varrho^\eta}d\varrho$. It is clear that \eqref{12} and \eqref{15} hold now for $r>0$, whence \begin{multline*}2\pi|\phi_\theta(z)|\leq\left|\int_{r\TT}\alpha_{z,\theta}\right|+|e^{-2\pi iz}-1|\left|\int_{\gamma_{r,\theta}}\alpha_{z,\theta}\right|\\\leq C_rr^{-\re z}e^{3\pi|\im z|}+2e^{2\pi|\im z|}M'e^{\theta\im z}r^{-\re z-1}\leq C_{\{r\}}r^{-\re z}e^{3\pi|z|},\quad r>0,\end{multline*} for constants $C_r,C_{\{r\}}$ depending only on $r$. This leads to $$|\phi_\theta(Re^{i\wi\theta})|\leq\frac{1}{2\pi}C_{\{r\}}r^{-R\cos\wi\theta}e^{3\pi R}=\frac{1}{2\pi}C_{\{r\}}\exp(-R\cos\wi\theta\log r+3\pi R)$$ and $$h_{\phi_\theta}(\wi\theta)\leq-\cos\wi\theta\log r+3\pi.$$ Passing with $r$ to infinity we arrive at $h_{\phi_\theta}(\wi\theta)=-\infty$.
\end{proof}
\bigskip\bigskip
Let $\dist$ denote the Euclidean distance. We have the following theorems.
\begin{tw}[Keldysh, \cite{Mer}, Chapter 2, \S 3, Theorem 1.3]\label{tw6}
Let $$E:=\{z\in\CC:\dist(z,[0,\infty))\geq 1\}$$ and let $f\in\OO(E)$. Then for any $\eps>0$ and $\eta\in(0,1/2)$ there exists an entire function $g$ such that \begin{equation}\label{25}
|f(z)-g(z)|\leq\eps e^{-|z|^\eta},\quad z\in E.\end{equation}
\end{tw}
\bigskip
\begin{tw}[Alexander, \cite{Ale}]\label{tw4} Let $\gamma\su\CC^n$ be a $\cC^1$-smooth proper imbedding of $\RR$ $($that is an injective $\cC^1$-smooth curve beginning and ending at infinity$)$ and let functions $h:\gamma\lon\CC$,\, $\eps:\gamma\lon(0,\infty)$ be continuous. Then there exists a holomorphic function $g:\CC^n\lon\CC$ such that $$|g-h|<\eps\text{ on }\gamma.$$
\end{tw}
The last one allows us to prove
\begin{lem}\label{4} Let $\gamma_1,\gamma_2\su\CC$ be disjoint closed half-lines. Let $\eps$ be a positive continuous function on $\gamma_1\cup\gamma_2$ and let $g\in\OO(\CC)$. Then there exist entire functions $g_1,g_2$ such that $$g=g_1+g_2,$$ $$|g_j|<\eps\text{ on }\gamma_j,\ j=1,2.$$
\end{lem}
\begin{proof} 
Let $\gamma\su\CC$ be a $\cC^1$-smooth injective curve such that $\gamma_1\cup\gamma_2\su\gamma$. There is a continuous function $h:\gamma\lon\CC$ such that $h=0$ on $\gamma_1$ and $h=g$ on $\gamma_2$. From Theorem \ref{tw4} there exists an entire function $g_1$ such that $|g_1-h|<\eps$ on $\gamma$. Therefore, functions $g_1$ and $g_2:=g-g_1$ satisfy the claim.
\end{proof}
\bigskip\bigskip
We have prepared all tools for the {\bf proof of Theorem \ref{tw3} (necessity)}:

Let $g$ be the function from Theorem \ref{tw6}, suitable for $f$. By Lemma \ref{4} applied to the half-lines $\gamma_{1,\theta_1}$, $\gamma_{1,\theta_2}$, where $\theta_1\neq\theta_2$, and functions $\zeta\longmapsto e^{-|\zeta|^\eta}$ and $g$, we get entire functions $g_1,g_2$ such that $$g=g_1+g_2,$$ $$|g_j(\zeta)|<Me^{-|\zeta|^\eta},\quad\zeta\in\gamma_{0,\theta_j},\ j=1,2,$$ with some constant $M$.

From Lemma \ref{9}, the corresponding functions $\phi_{\theta_j}$ for $g_j$ satisfy the conditions $(iv),(v),(vi),(viii)$ of Theorem \ref{tw3} for $g_j$ and $$h_{\phi_{\theta_j}}(\wi\theta)=-\infty,\quad|\wi\theta|<\pi/2.$$ Hence the function $\phi_g:=\phi_{\theta_1}+\phi_{\theta_2}$ satisfies the conditions $(iv),(v),(vi),(viii)$ of Theorem \ref{tw3} for $g$, instead of $f$, and $$h_{\phi_g}(\wi\theta)=-\infty,\quad|\wi\theta|<\pi/2.$$

Suppose that some function $\phi_{f-g}$ satisfies the necessary condition of Theorem \ref{tw3} for $f-g$. Then $\phi_f:=\phi_{f-g}+\phi_{g}$ satisfies the necessary condition of Theorem \ref{tw3} for $f$ (to see the condition $(vii)$ we use the property \eqref{6}). 

Concluding, it suffices to prove that there exists a function $\phi_{f-g}$ satisfying the necessary condition of Theorem \ref{tw3} for $f-g$. 

The inequality \eqref{25} implies that that for any $\delta\in(0,\pi)$ there is a constant $M_\delta$ depending only on $\delta$, such that the function $F:=f-g\in\OO(\cd)$ satisfies \begin{equation*}|F(\zeta)|\leq M_\delta e^{-|\zeta|^\eta},\quad\zeta\in\CC\sm\inte\Delta(-\delta,\delta),\end{equation*} \begin{equation}\label{22}|F(\zeta)|\leq M_\delta,\quad |\zeta|\leq e^{-\delta}.\end{equation} 

By Lemma \ref{9} used to the function $F$ and parameters $r\in(0,1)$,\, $\theta:=-\pi$, the formulas $$\phi_{r,\pi}(z):=\frac{1}{2\pi i}\int_{\gamma_{r,-\pi}^-\cup r\TT}\underbrace{e^{(-z-1)\Log_{-\pi}\zeta}F(\zeta)}_{=:\alpha_{z}(\zeta)}d\zeta+\frac{1}{2\pi i}\int_{\gamma_{r,-\pi}}\underbrace{e^{(-z-1)(\Log_{-\pi}\zeta+2\pi i)}F(\zeta)}_{=:\beta_{z}(\zeta)}d\zeta$$ define one and the same entire function $\phi$ satisfying the conditions $(iv),(v),(vi),(viii)$ of Theorem \ref{tw3} for $F$. The only task is to estimate the indicator function.

Note that $\Arg_{-\pi}$ and $\Log_{-\pi}$ restricted to $\CC\sm(-\infty,0]$ coincide with the standard main argument and logarithm. In particular, $$\Arg_{-\pi}e^{\pm it}=\Arg_{\pm t}e^{\pm it}=\pm t,\quad t\in(0,\pi).$$

For $r\in(0,1)$ and $\delta\in(0,\pi)$ consider the boundary of the domain $$D_r^\delta:=(\CC\sm\Delta(-\delta,\delta))\cup r\DD,$$ negatively oriented w.r.t. $D_r^\delta$, that is a curve $(\pa D_r^\delta)^-$ consisting of three pieces 
\begin{itemize}
\item $\gamma_{r,\delta}^-$;
\item $(r\TT\cap\Delta(-\delta,\delta))^-:=r\TT\cap\Delta(-\delta,\delta)$ negatively oriented w.r.t. $r\DD$;
\item $\gamma_{r,-\delta}$.
\end{itemize}
We claim that for any $z\in\CC$ the number \begin{equation}\label{17}2\pi i\phi(z)+\int_{(\pa D_r^\delta)^-}\alpha_{z}=\int_{\gamma_{r,-\pi}^-\cup r\TT}\alpha_{z}+\int_{\gamma_{r,-\pi}}\beta_{z}+
\int_{\gamma_{r,\delta}^-\cup(r\TT\cap\Delta(-\delta,\delta))^-\cup\gamma_{r,-\delta}}\alpha_{z}\end{equation} is equal to $$\lim_{R>0,\ R\to\infty}\int_{R\TT\cap D_r^\delta}\alpha_{z}$$ with $R\TT\cap D_r^\delta$ positively oriented w.r.t. $R\DD$. Actually, the above sum (over a path showed on Figure 1) is a limit as $R\to\infty$ of the integrals over the curves intersected with $R\CDD$.

\begin{figure}[ht]\includegraphics[scale=0.75]{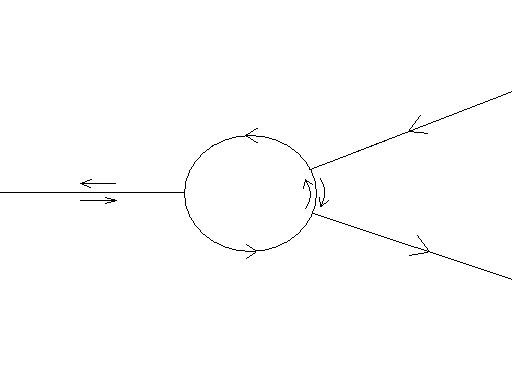}\caption{Path from the right side of \eqref{17}}\end{figure}

Then we approximate a path consisting of the following eight curves (Figure 2): 
\begin{itemize}
\item $\gamma_{r,-\pi}^-\cap R\CDD$;
\item $r\TT$;
\item $\gamma_{r,-\pi}\cap R\CDD$;
\item $R\TT\cap D_r^\delta\cap\{\im\zeta\geq 0\}$ negatively oriented w.r.t. $R\DD$;
\item $\gamma_{r,\delta}^-\cap R\CDD$;
\item $(r\TT\cap\Delta(-\delta,\delta))^-$;
\item $\gamma_{r,-\delta}\cap R\CDD$;
\item $R\TT\cap D_r^\delta\cap\{\im\zeta\leq 0\}$ negatively oriented w.r.t. $R\DD$;
\end{itemize}
by closed paths contained in $D_r^\delta$ (Figure 3) and apply the Cauchy Theorem for these paths and the function $\alpha_{z}\in\OO(D_r^\delta\sm(-\infty,0])$.

\begin{figure}[ht]\includegraphics[scale=0.75]{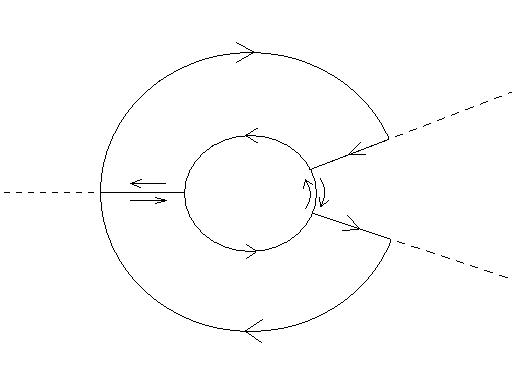}\caption{Approximated path (drawn by continuous lines)}\end{figure}\begin{figure}[ht]\includegraphics[scale=0.75]{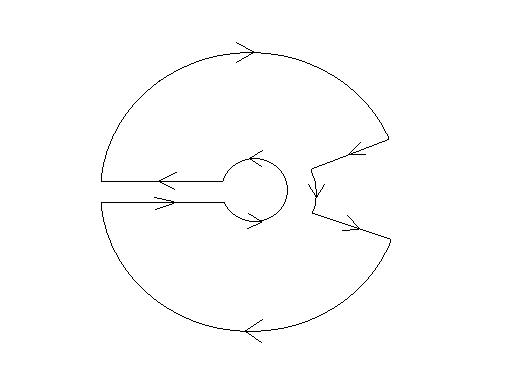}\caption{Approximating path}\end{figure} 

On the other side, for $z\in\CC$ \begin{multline*}\left|\int_{R\TT\cap D_r^\delta}\alpha_{z}\right|\leq
\int_{(-\pi,-\delta)\cup(\delta,\pi)}|e^{(-z-1)(\log R+it)}F(Re^{it})iRe^{it}|dt\\\leq
M_\delta e^{-R^\eta}R^{-\re z}\int_{(-\pi,-\delta)\cup(\delta,\pi)}e^{t\im z}dt\\\leq 2\pi e^{\pi|\im z|}M_\delta e^{-R^\eta}R^{-\re z}\to 0,\quad R\to\infty.\end{multline*}

So far, we have obtained $$\phi(z)=-\frac{1}{2\pi i}\int_{(\pa D_r^\delta)^-}\alpha_{z},\quad z\in\CC,\ r\in(0,1),\ \delta\in(0,\pi).$$ 

In what follows, let $r:=e^{-\delta}$, $\delta\in(0,\pi)$. 

We estimate for $z\in H$ \begin{eqnarray*}\left|\int_{\gamma_{r,\delta}^-}\alpha_{z}\right|&\leq&\int_r^\infty|e^{(-z-1)(\log\varrho+i\delta)}F(\varrho e^{i\delta})e^{i\delta}|d\varrho\\&\leq&\int_r^\infty\varrho^{-\re z-1}e^{\delta\im z}M_\delta e^{-\varrho^\eta}d\varrho\\&\leq&M_\delta e^{\delta\im z}\left(\int_r^1\varrho^{-\re z-1}e^{-\varrho^\eta}d\varrho+\int_1^\infty\varrho^{-\re z-1}e^{-\varrho^\eta}d\varrho\right)\\&\leq&
M_\delta e^{\delta\im z}\left(\int_r^1\varrho^{-|z|}\varrho^{-1}d\varrho+
\int_1^\infty e^{-\varrho^\eta}d\varrho\right)\\&\leq&M_\delta e^{\delta\im z}\left(\int_r^1r^{-|z|}\varrho^{-1}d\varrho+C\right)\\&=&M_\delta e^{\delta\im z}\left(r^{-|z|}\log(1/r)+C\right)\end{eqnarray*} with a constant $C$.
Analogously $$\left|\int_{\gamma_{r,-\delta}}\alpha_{z}\right|\leq M_{-\delta}e^{-\delta\im z}\left(r^{-|z|}\log(1/r)+C\right).$$
This gives $$\left|\int_{\gamma_{r,\delta}^-\cup\gamma_{r,-\delta}}\alpha_{z}\right|\leq M_\delta e^{\delta\im z}(e^{\delta|z|}\delta+C)+M_{-\delta}e^{-\delta\im z}(e^{\delta|z|}\delta+C)\leq K_\delta e^{2\delta|z|}$$ for $z\in H$ and some constant $K_\delta$ depending only on $\delta$. 

For the integral over the arc, we proceed as follows (recall that the inequality \eqref{22} holds) \begin{multline*}\left|\int_{r\TT\cap\Delta(-\delta,\delta)}\alpha_{z}\right|\leq \int_{-\delta}^{\delta}|e^{(-z-1)(\log r+it)}F(re^{it})ire^{it}|dt\leq M_\delta r^{-\re z}\int_{-\delta}^{\delta}e^{t\im z}dt\\\leq M_\delta e^{\delta\re z}2\delta e^{\delta|\im z|}\leq 2\delta M_\delta e^{2\delta|z|},\quad z\in\CC.\end{multline*} 

Finally, $$|\phi(z)|=\frac{1}{2\pi}\left|\int_{(\pa D_r^\delta)^-}\alpha_{z}\right|\leq L_\delta e^{2\delta|z|},\quad z\in H,$$ for a constant $L_\delta$ depending only on $\delta$, hence $$h_\phi(\wi\theta)\leq 2\delta,\quad|\wi\theta|<\pi/2.$$ Taking $\delta\to 0$, we finish the proof.\qed

\bigskip\bigskip\bigskip
\textsc{Acknowledgements.} I would like to thank Professor Pawe\l\ Doma\'nski for scientific care during writing this paper.

\end{document}